\documentclass[12pt, a4paper]{amsart}
\usepackage[hmargin=30mm, vmargin=25mm, includefoot, twoside]{geometry}
\usepackage[bookmarksopen=true]{hyperref}

\usepackage[alphabetic]{amsrefs}

\usepackage{amsfonts,amssymb,verbatim}
\usepackage{latexsym}
\usepackage{mathrsfs}
\usepackage{stmaryrd}
\usepackage{xspace}
\usepackage{enumerate, paralist}
\usepackage{graphicx}
\usepackage[all,cmtip]{xy}
\usepackage[all]{xypic}

\usepackage[usenames,dvipsnames]{color}
 \usepackage{txfonts, pxfonts}

\newtheorem{ThmIntro}{Theorem}

\newtheorem{thm}{Theorem}[section]
\newtheorem{cor}[thm]{Corollary}
\newtheorem{lem}[thm]{Lemma}

\newtheorem{prop}[thm]{Proposition}
\theoremstyle{definition}
\newtheorem{defn}[thm]{Definition}

\newtheorem{que}[thm]{Question}
\newtheorem{conj}[thm]{Conjecture}
\theoremstyle{remark}

\newtheorem{ex}[thm]{Example}

\numberwithin{equation}{section}

\newcommand{\Z}{\mathbb{Z}}
\newcommand{\G}{W}

\newcommand{\N}{\mathbb{N}}

\newcommand{\Aut}{\text{Aut}}

\newcommand{\HH}{\mathcal{H}}

\newcommand{\eps}{\varepsilon}

\newcommand{\diam}{\textnormal{diam}\,}
\newcommand{\gi}{\textnormal{girth}\,}

\begin{document}
\title{Admitting a coarse embedding is not preserved under group extensions}

\author{Goulnara Arzhantseva}
\address{Universit\"at Wien, Fakult\"at f\"ur Mathematik\\
Oskar-Morgenstern-Platz 1, 1090 Wien, Austria.}
\email{goulnara.arzhantseva@univie.ac.at}

\author{Romain Tessera}
\address{Laboratoire de Math\'ematiques d'Orsay, Univ. Paris-Sud, CNRS, Universit\'e
Paris-Saclay, 91405 Orsay, France.
}
\email{romain.tessera@math.u-psud.fr}
\date{}
\subjclass[2010]{46B85, 20F69, 22D10, 20E22}
\keywords{Relative Kazhdan's property (T), Haagerup property, Gromov's a-T-menability, expander,  coarse embeddings, coarse amenability.}

\thanks{The research of G.A.\ was partially supported by the ERC grant ANALYTIC no.\ 259527. The research of R.T.\ was partially supported by 
the ANR Blanc ANR-14-CE25-0004-01, acronym GAMME} 
\baselineskip=16pt

\begin{abstract}
We construct a finitely generated  group which is an extension of two finitely generated groups coarsely embeddable into Hilbert space
but which itself does not coarsely embed into Hilbert space.
Our construction also provides a new infinite monster group:
the first example of a finitely generated group that does not coarsely embed into Hilbert space and yet does not contain a weakly embedded expander.
\end{abstract}
\maketitle

\section{Introduction}
We answer, in the negative, the following well-known question~\cite{DG}, see also~\cite[Section 3]{GK}:
\medskip

\noindent
\emph{Given two finitely generated groups coarsely embeddable into Hilbert space, does their extension necessary coarsely embed into Hilbert space?}\smallskip

The concept of coarse embedding was introduced by Gromov~\cite[p. 218]{Gr_ai} in his investigation of the Novikov conjecture (1965) on the homotopy invariance of higher signatures for closed manifolds.

\begin{defn}[Coarse embedding]
Let $(X_n,d_n)_{n\in \N}$ be a sequence of metric spaces and let $(Y, d_Y)$ be a metric space.
A sequence of maps $f_n\colon X_n\to Y$ is a \emph{ coarse embedding of $(X_n)_{n\in \N}$ into $Y$} if there exist two 
proper functions $\rho, \gamma\colon [0,\infty)\to [0,\infty)$ such that for all $n\in \N$ and all $x,y\in X_n$,
$$\rho(d_n(x,y))\leqslant d_Y\left(f_n(x),f_n(y)\right)\leqslant \gamma(d_n(x,y)).$$
\end{defn}
Applied to a constant sequence $X_n=G$, where $G$ is a countable discrete group  equipped with a proper left invariant metric (e.g. a finitely generated group endowed with the word length metric), this definition gives the notion of a group coarsely embeddable into a metric space. Every two proper left invariant metrics on a countable discrete group are known to be \emph{coarsely equivalent} (i.e. the two metric spaces admit coarse embeddings into each other). Therefore, the property that $G$ coarsely embeds into a metric space $Y$ does not depend on the choice of such a metric on the group. 
In this paper, we focus on the case where $Y$ is an infinite dimensional Hilbert space.

A countable discrete group coarsely embeddable into a Hilbert space satisfies the coarse Baum-Connes conjecture by a remarkable result of Yu~\cite{Yu}.  
Recently, Kasparov and Yu established the strong Novikov conjecture with coefficients for every countable group
coarsely embeddable into a Banach space from a wide class of Banach spaces, including Hilbert spaces, $\ell^p$-spaces,
and uniformly convex Banach spaces with certain unconditional bases~\cite{KYu}.
These and other related deep results have generated
an intense study of groups and metric spaces which are coarsely embeddable into Hilbert and Banach spaces.

Coarse embeddability of countable discrete groups into Hilbert space is usually investigated in comparison with the following properties:
 von Neumann's amenability; Haagerup property, also known as Gromov's a-T-menability;  
coarse amenability, equivalent to the existence of a topological amenable action of the group on a compact Hausdorff space (or 
to the exactness of the reduced $C^*$-algebra of the group). The implications between these properties are summarized as follows, see, for instance,~\cite{NYu}:
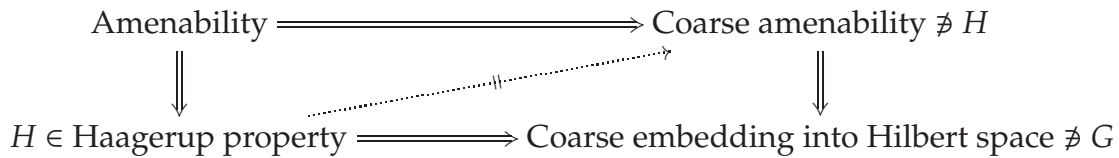
\[
\xymatrixcolsep{5pc}\xymatrix{
\hbox{Amenability} \ar@{=>}[d] \ar@{=>}[r] 
& \hbox{Coarse amenability} \ar@{=>}[d] \not\ni H\\
H\in\hbox{Haagerup property}  \ar@{.>}[ru] |{\SelectTips{cm}{}\object@{||}}|{}  \ar@{=>}[r] & \hbox{Coarse embedding into Hilbert space} \not\ni G}
\]

This diagram is rather informative and is at the origin of a lot of  intensive research.
Amenability and the Haagerup property
require in their definitions an equivariance under the group action on itself by left multiplication.
The right column properties are obtained by ``forgetting'' this equivariance.
Amenability and coarse amenability have characterizations that involve
the existence of appropriate functions with finite support. The Haagerup property and coarse embeddability 
relax this finite support condition into the support vanishing at infinity. Neither of the indicated implications can be reversed.

Permanence properties of these classes of groups are subject of thorough  study.
In particular, it is well-known that amenability and coarse amenability are preserved under taking group extensions~\cite{KW}.
In contrast, famous examples of semi-direct products with relative Kazhdan's property (T), such as $\Z^2\rtimes SL_2(\Z)$ relative to $\Z^2$, show that, in general,
the class of groups with the Haagerup property is not closed under taking group extensions. 
Therefore,  whether or not coarse embeddability into Hilbert space survives while taking group extensions, is an intriguing question.  
Several positive results are known. Cyclic extensions of torsion-free $C'(1/8)$ small cancellation groups~\cite{Sela},
HNN-extensions  and extensions of coarsely embeddable groups by coarsely amenable groups~\cite{DG},
restricted regular wreath products of coarsely embeddable groups  and 
some restricted permutational wreath products of coarsely embeddable groups~\cite{CSV, Li} and metric spaces~\cite{CD} are known to coarsely embed into Hilbert space.

Our main result is the following.  

\begin{ThmIntro}\label{thm:main}
There exists a finitely generated group which is a split extension of an (infinite rank) abelian group by a finitely generated group with the Haagerup property that does not coarsely embed into Hilbert space.
 \end{ThmIntro}

We specifically show that a certain restricted permutational wreath product $\Z/2\Z\wr_{G}H$ does not coarsely embed 
into Hilbert space. Here, $G$ is a finitely generated group with no coarse embedding into Hilbert space and
$H$ is a finitely generated group with the Haagerup property\footnote{Actually, this group has an even stronger property: it acts properly on a CAT(0) cubical complex.} which is not coarsely amenable and such that
there is a surjective homomorphism $H \twoheadrightarrow G$ (that induces the action of $H$ on $G$ by left translation and defines the wreath product). 

The group $\Z/2\Z\wr_{G}H$ does not yet answer the above well-known question~\cite{DG, GK}, as the kernel of its surjection onto $H$ is an infinite rank abelian torsion group which is not finitely generated. However, we modify this example to produce a required extension of two \emph{finitely generated} groups. 
Namely, we consider a surjective homomorphism $F\twoheadrightarrow G$ from a finitely generated free group $F$ onto $G$, and form the restricted permutational wreath product $\Z/2\Z\wr_{G}(H\times F)$, where $G$ is the $H\times F$-set, with $H$ acting as above by left translation via its surjection onto $G$, and $F$ acting by right translation via its surjection onto $G$. Since these two actions commute, this indeed gives an $H\times F$ action on $G$. Thus, $\Z/2\Z\wr_{G}(H\times F)$ is the extension of $\Z/2\Z\wr_{G} F$ by $H$ and both groups are finitely generated. Now $\Z/2\Z\wr_{G}F$ is coarsely amenable, being an extension of two coarsely amenable groups \cite{AR}. In particular, $\Z/2\Z\wr_{G}F$ coarsely embeds into a Hilbert space. We therefore obtain
\begin{ThmIntro}\label{thm:main'}
There exists a finitely generated group which is a split extension of a finitely generated coarsely amenable group by a finitely generated group with the Haagerup property that does not coarsely embed into Hilbert space.
 
 In particular, coarse embeddability into Hilbert space is not stable under taking group extensions among finitely generated groups. 
 \end{ThmIntro}

Concrete groups $G$ and $H$ which are used for our construction are chosen among infinite monsters.
 $G$ is the  \emph{special Gromov monster}: a finitely generated group which
contains an expander graph (see Definition~\ref{def:ex}) isometrically in its Cayley graph;  
$H$ is the \emph{Haagerup monster}: a finitely generated group which satisfies the Haagerup property  but 
is not coarsely amenable, chosen so that $H$ contains a certain finite-sheeted covering of a suitable expander graph isometrically in its Cayley graph; 
see Section~\ref{sec:graphical} for comments on the existence of such groups.

Our result yields additional counterexamples to Conjecture~1.4 from \cite{CSV}, where the authors 
expected that the Haagerup property is preserved under taking restricted permutational wreath products.
This conjecture was first negated  in \cite[Corollary 3.4]{CI}  by Chifan-Ioana whose results
immediately imply that $\Z/2\Z\wr_{G}H$ has not the Haagerup property or, in other words, has no \emph{equivariant} coarse embeddings into Hilbert space.
 Theorem~\ref{thm:main} gives a counterexample also to the non-equivariant counterpart of this conjecture. 

In proving Theorem~\ref{thm:main}, we also answer, in the affirmative, the following open problem~\cite[Section 6]{GK}, see also~\cite[Section 5.7]{NYu} and~\cite[Section 8]{AT}:
\medskip

\noindent
\emph{Does there exist a finitely generated group which does not coarsely embed into Hilbert space and yet has no weakly embedded expander?}\smallskip

\begin{defn}[Weak embedded expander]
Let $(X_n,d_n)_{n\in \N}$ be an expander graph and let $(Y, d_Y)$ be a metric space.
A sequence of maps $f_n\colon X_n\to Y$ is a \emph{weak embedding of $(X_n)_{n\in \N}$ into $Y$}
if there exists $D>0$ such that $f_n$ are $D$-Lipschitz and
$$
\lim_{n\to\infty}\sup_{x\in X_n}\frac{|f_n^{-1}(f_n(x))|}{|X_n|} =0.
$$
\end{defn}

To admit a weakly embedded expander is a well-known obstruction for a metric space to coarsely embed into Hilbert space~\cite{M,Gr_sp,Gr_rw}.
It is also crucial in Higson-Lafforgue-Skandalis work to refute the Baum-Connes conjecture with coefficients~\cite{HLS}. 

In a recent paper~\cite{AT}, using the notion of relative expansion, we construct many regular \emph{graphs} which do not coarsely embed into Hilbert space and admit no weakly embedded expanders. This gives the first examples of graphs with bounded degree not coarsely embeddable into Hilbert space and yet having no weakly embedded expander. 
The proof that the group $\Z/2\Z\wr_{G}H$ does not coarsely embed into Hilbert space relies on the fact that it contains an embedded sequence of such relative expanders. As a by-product, we obtain 
the first example of a \emph{group} which does not coarsely embed into Hilbert space and yet admits no weakly embedded expanders. Indeed, this follows from the fact that an extension of two groups which are coarsely embeddable into Hilbert space does not contain weakly embedded expanders~\cite[Proposition~2]{AT}. Our group $\Z/2\Z\wr_{G}H$ is such an extension.

\begin{ThmIntro}\label{thm:mainbis}
There exists a finitely generated group that does not coarsely embed into Hilbert space and yet does not contain weakly embedded expanders. 
\end{ThmIntro}

\subsection*{Acknowledgment}  We are grateful to the Erwin Schr\"odinger Institute for Mathematics and Physics for hospitality 
 during ``Measured group theory'' program in 2016.  We thank Damian Sawicki for signaling a point that we had overlooked in a previous version\footnote{Namely, that the graphs $Cay(\G_n,\Sigma_n)$ may not necessarily embed isometrically into $Cay(\G,\Sigma)$, cf. Corollary~\ref{cor:subgraph}.}.

\section{Graphs and infinite coarse monsters}
\subsection{Graphs and graphical presentations}\label{sec:graphical}
Let $S$ be a set and $S^{\pm }=S\sqcup S^{-}$, where $S^-$ denotes the set of formal  inverses $\{s^{-1} \mid s\in S\}$.
A graph $\Gamma$ is \emph{$S$-labeled} if each edge has an orientation and each oriented edge has a label $s\in S^{\pm }$ so that the labels of opposite edges are formal inverses.
A label of a path in an $S$-labeled graph is, by definition, a word in letters from $S^{\pm }$ which is 
the concatenation of labels of the edges along the path.  A labelling is \emph{reduced} if the labels of reduced paths are freely reduced words.

\begin{defn}[Graphical presentation of a group]\label{def:group} The group defined by an  $S$-labeled graph $\Gamma$ is the
group defined by the \emph{graphical presentation} $$G(\Gamma)=\langle S \mid R\rangle,$$
where $R$ denotes the set of labels of all non-trivial simple closed paths in $\Gamma$.
\end{defn}

Every $S$-generated group possesses such a graphical presentation: 
$\Gamma$ can be taken to be a disjoint union of cycle graphs, $S$-labelled by the relator words.

The following observation is useful for our construction.

\begin{lem}[Quotients from coverings]\label{lem:cover}
Let $p\colon \widetilde\Gamma\to\Gamma$ be a covering of $S$-labeled graphs, i.e. a label-preserving graph covering.
Then there exists a  surjective homomorphism $G(\widetilde\Gamma)\twoheadrightarrow G(\Gamma)$ with the identity map $s\mapsto s$ on the generators.
\end{lem}
\begin{proof}
This is a graph covering so we have $\pi_1(\widetilde\Gamma)\stackrel{\ p_*}\hookrightarrow\pi_1(\Gamma).$ 
Since the covering is label-preserving, it follows from Definition~\ref{def:group}, that all the relator words of $G(\widetilde\Gamma)$
represent the identity element in $G(\Gamma)$. By classical von Dyck's theorem, the required surjection is immediate.
\end{proof}

We view a graph $\Gamma$ as a sequence of its connected components $\Gamma=(\Gamma_n)_{n\in\mathbb{N}}$ each
of which is equipped with the edge-length metric.
We focus on two specific graphs which are used in the construction of the above-mentioned infinite monster groups $G$ and $H$: an expander and a graph with walls, respectively.

We begin with some auxiliary definitions. Recall that the \emph{girth} of a graph is the edge-length of its shortest non-trivial cycle and
the \emph{diameter} of a graph is the greatest edge-length distance between any pair of its vertices.

\begin{defn}[{\rm dg}-bounded graph]
We say that a graph $\Gamma=(\Gamma_n)_{n\in\mathbb{N}}$ is \emph{dg-bounded} if there exists a constant $D>0$ such that  for all $n\in\mathbb N$:
$$\frac{\diam \Gamma_n}{\gi \Gamma_n}\leqslant D.$$ 
\end{defn}
If $\Gamma$ is with uniformly bounded degree, say at most $d>0$, 
then the number of vertices $\vert \Gamma_n\vert$ satisfies  $\vert \Gamma_n\vert \leqslant 1+d+\cdots+d(d-1)^{\diam \Gamma_n -1}.$ 
If, in addition,  $\Gamma$ is dg-bounded, then 
$\gi \Gamma_n \geqslant \frac{1}{D}\diam \Gamma_n \geqslant \log_{d-1}\vert \Gamma_n\vert  -\frac{2}{d}.$
Thus, a dg-bounded graph with uniformly bounded degree satisfies $\gi \Gamma_n\to \infty$ whenever  $\vert \Gamma_n\vert\to \infty$ as  $n\to\infty.$

In this paper, we consider graphs with uniformly bounded degree only.

\begin{defn}[Cheeger constant]
Given a finite connected graph $\Gamma$ with $\vert \Gamma\vert$ vertices and a subset $A\subseteq \Gamma$, denote by 
$\partial A$ the set of edges between $A$ and $\Gamma\setminus A$. 
The \emph{Cheeger constant} of $\Gamma$ is defined as $$h(\Gamma)=\min_{1\leqslant |A|\leqslant \vert \Gamma\vert/2}\frac{|\partial A|}{|A|}.$$ 
\end{defn}

\begin{defn}[Expander]\label{def:ex}
An \emph{expander} is a sequence $(\Gamma_n)_{n\in\mathbb{N}}$  of finite connected graphs with uniformly bounded degree, $\vert \Gamma_n\vert\to\infty$ as $n\to \infty$,  and
$h(\Gamma_n)\geqslant c$ uniformly over $n\in \mathbb{N}$ for some constant $c>0$.\smallskip
\end{defn}

The  next result is well-known, see, for example,~\cite[Proposition 11.29]{R}.
\begin{lem}\label{lem:ex}
No expander coarsely embeds into Hilbert space.
\end{lem}

Here is a celebrated example of a dg-bounded expander~\cite{Ma,LPS}.

\begin{ex}[dg-bounded expander]\label{ex:lps}
The Cayley graph $X^{p,q}$ of the projective general linear group $PGL_2(q)$ over the field of $q$ elements
for a particular set  of $(p + 1)$ generators, where $p$ and $q$ are distinct primes
congruent to 1 modulo 4 with the Legendre symbol $\left(\frac{p}{q}\right) = -1$, satisfies the following:
\begin{itemize}
\item[(1)] $X^{p,q}$ is $(p+1)$ regular on $N=q(q^2-1)$ vertices;
\item[(2)] $\gi X^{p,q}\geqslant  4\log_pq-\log_p4$;
\item[(3)] $(X^{p,q})_q$ is a family of Ramanujan graphs.
\end{itemize}
For a fixed prime $p$ and each $q$ as above, we set $\Gamma_q= X^{p,q}$. Then $(\Gamma_q)_q$ is an expander and
$\gi \Gamma_q\to\infty$ as $q\to\infty$. The expander mixing lemma ensures that $\Gamma_q$ is of diameter
$O(\log N)$. Thus, $(\Gamma_q)_q$ is a dg-bounded expander. Moreover, it is a dg-bounded expander whose components
are Cayley graphs.
\end{ex}

In contrast to expanders, are sequences of \emph{graphs with walls}: graphs where every edge belongs to a wall.
A \emph{wall} (or a \emph{cut}) in a connected graph is a collection of edges such that removing all open edges of the collection decomposes the graph into exactly two connected components. 

A distance between two vertices in the \emph{wall pseudo-metric} on a graph with walls is the number of walls separating these vertices.
Elementary examples of graphs with walls are: a tree and the 1-skeleton of the square tiling of the plane,
with the wall pseudo-metric being the edge-length metric.

The following result is well-known, see, for example,~\cite[Proposition 4.3]{AGS}.
\begin{lem}
Let $\Gamma=(\Gamma_n)_{n\in\mathbb{N}},$ where each $\Gamma_n$ be a graph with walls.
Then, viewed with respect to the wall pseudo-metric, $\Gamma=(\Gamma_n)_{n\in\mathbb{N}}$ admits a coarse embedding into Hilbert space.
\end{lem}

Graphs with walls are examples of spaces with walls. For a set $X$ and a family $\mathcal W$ of partitions, called \emph{walls}, of $X$ into two parts, 
the pair $(X, \mathcal W)$ is a \emph{space with walls} if every two distinct points of $X$ are separated by finitely many walls, i.e. the distance between distinct points in the
wall-pseudo metric is finite. A group acting properly isometrically on a space with walls satisfies the Haagerup property,
see e.g.~\cite[Section 1.2.7]{CCJ}.

A natural way to produce a graph with walls is to take the $\mathbb Z / 2\mathbb Z$-homology covering of a \emph{2-connected} graph (i.e. a graph where removing an edge does not disconnect the graph).
Such a countably successive covering of the figure eight graph (or of a bouquet of finitely many cycle graphs) yields a regular \emph{graph} which is coarsely embeddable into Hilbert space
but is not coarsely amenable~\cite{AGS}. One can also take the $\mathbb Z / 2\mathbb Z$-homology covering of each component of a given sequence $\Gamma=(\Gamma_n)_{n\in\mathbb{N}}$, cf.~\cite[Example 3 and below]{AO}.

\begin{ex}[$\mathbb Z/2\mathbb Z$-homology covering of a dg-bounded graph]
Let $\Gamma=(\Gamma_n)_{n\in\mathbb{N}}$ be a dg-bounded graph such that each component $\Gamma_n$ is a 2-connected graph. Let 
$\widetilde\Gamma=(\widetilde\Gamma_n)_{n\in\mathbb{N}}$ be the \emph{$\mathbb Z / 2\mathbb Z$-homology covering} of $\Gamma$:
each $\widetilde\Gamma_n$ is the regular covering of $\Gamma_n$ whose the group of deck transformations is the 
$\rm{rank}(\pi_1(\Gamma_n))$-fold direct sum of $\mathbb Z/2\mathbb Z$'s.
Then each $\widetilde \Gamma_n$ is the graph with walls~\cite[Lemma 3.3]{AGS}.
For each edge $e\in E(\Gamma_n)$ and the covering map $p\colon \widetilde\Gamma_n\to \Gamma_n$,
the wall $w_e$ is defined by $w_e=p^{-1}(e)\subseteq E(\widetilde\Gamma_n)$
and $\{w_e \mid e \in E(\Gamma_n)\}$ provides the \emph{wall structure} on $\widetilde \Gamma_n$ (meaning that each edge  is contained in exactly one wall).
\end{ex}

\subsection{Infinite coarse monster groups}
The dg-bounded graphs play a significant role in recent constructions of infinite \emph{groups} containing
prescribed graphs in their Cayley graphs. Indeed, dg-bounded graphs are particularly suited
to carry labelings with various  small cancellation conditions. Such labelings guarantee, in particular, 
the injectivity of the natural label-preserving graph morphism
$$
\Gamma_0 \to Cay(G(\Gamma),S),
$$
for each component $\Gamma_0$ of an $S$-labelled graph $\Gamma$,
after choosing a base vertex in $\Gamma_0$ and its image in the Cayley graph  $Cay(G(\Gamma),S)$
of the group it defines. This, together with the requirements of a given small cancellation condition, 
allows to induce desirable  coarse properties of the group $G(\Gamma)$ from those of the graph $\Gamma$.
Here is a brief outline of two constructions of this type.

Gromov has introduced the \emph{geometric small cancellation condition}~\cite[Section 2]{Gr_rw}, see also~\cite[Section 3]{AD}.
A labeling of a dg-bounded expander $\Gamma=(\Gamma_n)_{n\in\mathbb{N}}$ by a finite set of labels, chosen uniformly at random, satisfies the condition. 
This yields an almost quasi-isometric embedding of the dg-bounded expander into 
the Cayley graph $Cay(G(\Gamma),S)$ of the group it defines, and, also using Lemma~\ref{lem:ex},
\emph{Gromov's monster}: a finitely generated group admitting no coarse embeddings into Hilbert space~\cite{Gr_rw,AD}.

Arzhantseva-Osajda have introduced the \emph{lacunary walling condition}~\cite[Definition 4.1]{AO} on graphs with walls $\Lambda=(\Lambda_n)_{n\in\mathbb{N}}$, which includes the \emph{$C'(\overline \lambda)$-small cancellation condition} on the labeling. Under even a weaker variant of this small cancellation condition (cf. Definition~\ref{def:sc} below), they observe that
each component $\Lambda_0$ of $\Lambda$  embeds isometrically into the Cayley graph $Cay(G(\Lambda),S)$~\cite[Lemma 2.1]{AO}. Using the lacunary walling condition, they extend walls from each
component $\Lambda_n$ of $\Lambda=(\Lambda_n)_{n\in\mathbb{N}}$
to $Cay(G(\Lambda),S)$, prove that the wall pseudo-metric on $Cay(G(\Lambda),S)$ is bi-Lipschitz equivalent to the graph metric, and,
hence, in particular, the action of $G(\Lambda)$ on the resulting space with walls is proper~\cite[Theorem 1.1, Theorem 4.7, and Theorem 5.1]{AO}.
Applied to a suitable graph, this gives the existence of the \emph{Haagerup monster}\footnote{This terminology is to emphasize
the eccentricity of this group with the Haagerup property.}: a finitely generated group with the Haagerup property, which is not coarsely amenable.

\begin{thm}\cite[Theorem 1.2 and Theorem 5.1]{AO}\label{thm:AO}
Let $\Lambda=(\Lambda_n)_{n\in\mathbb{N}}$ be a uniformly bounded degree graph with all vertex degrees at least 3, satisfying the lacunary walling condition.
Then $G(\Lambda)$ acts properly on a space with walls (equivalently, acts properly on a CAT(0) cubical complex) but is not coarsely amenable.
In particular, $G(\Lambda)$ has the Haagerup property but is not coarsely amenable, i.e. $G(\Lambda)$ is the Haagerup monster.
\end{thm}

A concrete example of a graph $\Lambda=(\Lambda_n)_{n\in\mathbb{N}}$ with the lacunary walling condition is constructed by taking  the $\mathbb Z / 2\mathbb Z$-homology covering of a 
dg-bounded graph $\Gamma=(\Gamma_n)_{n\in\mathbb{N}}$ with $\gi \Gamma_n\to\infty$ as $n\to\infty,$
under the assumption that there exists a labeling of $\Gamma$ with the $C'(\overline \lambda)$-small cancellation condition~\cite[Definition 4.1 and Section 7]{AO}.

Here is a familiar weaker variant of the $C'(\overline \lambda)$-small cancellation condition, see, for instance,~\cite[Section 2]{AO}. A path is 
 a \emph{piece} in an $S$-labeled graph $\Gamma=(\Gamma_n)_{n\in\mathbb{N}}$ if 
 there are $\Gamma_n$ and $\Gamma_m$ containing this path
in two essentially distinct ways (i.e. not equal by a label-preserving graph isomorphism $\Gamma_n\to \Gamma_{m}$).

\begin{defn}[$C'(\lambda)$-small cancellation condition]\label{def:sc}
Given $\lambda\in (0,1),$
an $S$-labeled graph $\Gamma=(\Gamma_n)_{n\in\mathbb{N}}$ satisfies the \emph{$C'(\lambda)$-small cancellation condition}
if the labelling is reduced and every piece appearing in $\Gamma_n$ has length strictly less than $\lambda \gi \Gamma_n$.
\end{defn}

The next lemma is now standard, see, for instance,~\cite[Lemma 2.1]{AO}, where $\lambda\leqslant 1/24$ since the setting is more general; cf.~\cite[Theorem 1, (6)]{Oll}, where $\Gamma$ is assumed to be finite and~\cite[Section 3.2]{Dom}, where a weaker requirement on the length of pieces is explored.

\begin{lem}\label{lem:isom}
Let $\Gamma=(\Gamma_n)_{n\in\mathbb{N}}$ be an $S$-labeled graph with the $C'(\lambda)$-small cancellation condition for $\lambda\leqslant 1/6.$
Then, for each component $\Gamma_0$ of $\Gamma$, the label-preserving graph morphism
$
\Gamma_0 \to Cay(G(\Gamma),S)
$
is an isometric embedding.
\end{lem}

The following result was proved through probabilistic arguments by Osajda.

\begin{thm}\cite[Theorem 1]{Os}\label{thm:os}
Let  $\Gamma=(\Gamma_n)_{n\in\mathbb{N}}$ be a uniformly bounded degree dg-bounded graph with $\gi \Gamma_n\to\infty$ as $n\to\infty.$
Then, for every $\lambda>0$, there exists a $C'(\lambda)$-small cancellation labeling of a subsequence
 $(\Gamma_{n_k})_{k\in\mathbb{N}}$ of $\Gamma,$ over a finite set of labels\footnote{The probabilistic argument requires that $\lambda\gi\Gamma_n>1$ and $\gi\Gamma_n<\gi\Gamma_{n+1}$, hence passing to a subsequence. Many subsequences (and possibly $\Gamma$ itself) satisfy this mild assumption.}.
\end{thm}

Applied to a dg-bounded expander, this theorem gives, via Lemma~\ref{lem:isom}, the existence of the \emph{special Gromov monster}: 
a finitely generated group which contains in its Cayley graph an isometrically embedded expander~\cite[Theorem 4]{Os}.

Albeit the $C'(\lambda)$-small cancellation condition is weaker than the above-mention\-ed $C'(\overline\lambda)$-small cancellation condition,
combined with a \emph{proper} lacunary wall\-ing condition~\cite[Definition 5.1]{Os}, it gives, following arguments of~\cite[Section 4]{AO},
the coarse equivalence between the wall pseudo-metric on $Cay(G(\Lambda),S)$ and its graph metric.
Then, as in Theorem~\ref{thm:AO},  $G(\Lambda)$ acts properly on a space with walls and, in particular,
$G(\Lambda)$ has the Haagerup property~\cite[Theorem 5.6. and Theorem 3]{Os}. 

A concrete example of a graph $\Lambda=(\Lambda_n)_{n\in\mathbb{N}}$ with the proper lacunary walling condition is constructed as in~\cite[Section 7]{AO}
by taking a finite-sheeted covering (to enlarge the girth of the base graphs) followed by the $\mathbb Z / 2\mathbb Z$-homology covering (to produce the walls)
of a  dg-bounded graph $\Gamma=(\Gamma_n)_{n\in\mathbb{N}}$ with $\gi \Gamma_n\to\infty$ as $n\to\infty$~\cite[Lemma 6.1]{Os}.
The assumption that there exists a labeling of $\Gamma$ with the $C'(\lambda)$-small cancellation condition is now guaranteed by Theorem~\ref{thm:os}.
(We do not detail the lacunary walling conditions as we do not use them below.)    

Here is the outcome of~\cite{AO,Os} summarized in a suitable form.

\begin{thm}\label{thm:exist}
Let $\Gamma=(\Gamma_n)_{n\in\mathbb{N}}$ be a uniformly bounded degree dg-bounded graph with all vertex degrees at least 3
and $\vert \Gamma_n\vert\to \infty$ as  $n\to\infty.$
Then there exist $\lambda\in (0,1/24]$, a $C'(\lambda)$-small cancellation labeling of a subsequence $(\Gamma_{n_k})_{k\in\mathbb{N}}$ of $\Gamma,$ over a finite set $S$ of labels, and  
a graph $\Lambda=(\Lambda_{k})_{k\in\mathbb{N}}$ with a finite-sheeted regular label-preserving
graph covering $p \colon (\Lambda_{k})_{k\in\mathbb{N}} \to (\Gamma_{n_k})_{k\in\mathbb{N}}$ so that
$G(\Lambda)$ is the Haagerup monster.
\end{thm}

Observe that the labeling of a covering graph induced from a reduced labeling of the base graph is reduced (e.g. by the unique lifting property of the covering spaces).  Also, the labeling from Theorem~\ref{thm:os} actually restricts the length of a wider family of pieces in $\Gamma'=(\Gamma_{n_k})_{n\in\mathbb{N}},$ that of graphs appearing in two distinct ways
(not only of paths that occur in two essentially distinct ways). Taking a covering then induces the  $C'(\lambda)$-small cancellation labeling of $\Lambda=(\Lambda_{k})_{k\in\mathbb{N}}$,
in the sense of Definition~\ref{def:sc}. This is because essentially distinct (not equal by a deck transformation) paths in $\Lambda$ project under the graph covering        
onto distinct paths in $\Gamma'$, cf. \cite[Remark 1.12]{Dom}. Then Lemma~\ref{lem:cover}, Lemma~\ref{lem:isom},  the fact that the lengths of pieces remain the same in coverings, and the preceding theorem, applied to a dg-bounded expander, give:

\begin{prop}\label{prop:surj}
Let $\Gamma=(\Gamma_n)_{n\in\mathbb{N}}$ be a dg-bounded expander, then
there exist a reduced $S$-labeling of a subsequence $\Gamma'=(\Gamma_{n_k})_{k\in\mathbb{N}}$ of $\Gamma,$ for a finite set $S,$ 
and a graph $\Lambda=(\Lambda_{k})_{k\in\mathbb{N}}$ with a finite-sheeted regular label-preserving
graph covering $p \colon (\Lambda_{k})_{k\in\mathbb{N}} \to (\Gamma_{n_k})_{k\in\mathbb{N}}$ so that:
\begin{itemize}
\item $G:=G(\Gamma')$ is the special Gromov monster;
\item $H:=G(\Lambda)$ is the Haagerup monster;
\item For each component $\Gamma_0$ of $\Gamma'$, the label-preserving graph morphism
$
\Gamma_0 \to Cay(G,S)
$
is an isometric embedding;
\item For each component $\Lambda_0$ of $\Lambda$, the label-preserving graph morphism
$
\Lambda_0 \to Cay(H,S)
$
is an isometric embedding;
\item  There is a surjective homomorphism $H\twoheadrightarrow G$ 
with the identity map $s\mapsto s$ on the generators.
\end{itemize}
\end{prop}

If $\Gamma=(\Gamma_n)_{n\in\mathbb{N}}$ is a dg-bounded expander whose components are Cayley graphs,
then we can ensure that the components of $\Lambda=(\Lambda_{k})_{k\in\mathbb{N}}$ are Cayley graphs as well. 

\begin{prop}\label{prop:surjC}
Let $\Gamma=(\Gamma_n)_{n\in\mathbb{N}}$ be a dg-bounded expander with $\Gamma_n=Cay(G_n,S_n)$ for a group $G_n$ generated by $S_n$, then
there exist a reduced $S$-labeling of a subsequence $\Gamma'=(\Gamma_{n_k})_{k\in\mathbb{N}}$ of $\Gamma,$ for a finite set $S,$ 
and a graph $\Lambda=(\Lambda_{k})_{k\in\mathbb{N}}$ with a finite-sheeted regular label-preserving, with respect to $S$,
graph covering $p \colon (\Lambda_{k})_{k\in\mathbb{N}} \to (\Gamma_{n_k})_{k\in\mathbb{N}}$ such that all the conclusions of Proposition~\ref{prop:surj} hold and, in addition,
\begin{itemize}
\item $\Lambda_k=Cay(H_k,S_{n_k})$ for a group $H_k$ generated by $S_{n_k}$, for each $k\in\mathbb{N}$;
\item  There is a surjective homomorphism $H_k\twoheadrightarrow G_{n_k}$ 
with the identity map on the generators $S_{n_k}$;
\item The $S$-labeled graph $\Gamma_{n_k}$ (respectively, the $S$-labeled graph $\Lambda_k$) is isometric to
 $Cay(G_{n_k},S_{n_k})$ (respectively, to $Cay(H_k,S_{n_k})$).
\end{itemize}
\end{prop}

\begin{proof}
A covering is regular if and only if the fundamental group of the covering graph is a normal subgroup in the fundamental group of the base graph.
A graph is a Cayley graph (with respect to $\ell$ generators) if and only if it is a regular graph covering of a bouquet  (of $\ell$ cycle graphs).
Since $\Gamma_n$ is a Cayley graph, there exists a regular graph covering $\Gamma_n\to \Omega_n$, where $\Omega_n$ is the bouquet of $\vert S_n\vert$ cycle graphs.
Let $p \colon (\Lambda_{k})_{k\in\mathbb{N}} \to (\Gamma_{n_k})_{k\in\mathbb{N}}$ be a finite-sheeted regular graph covering given by Proposition~\ref{prop:surj}.
We have $\pi_1(\Lambda_{k})\unlhd\pi_1(\Gamma_{n_k})\unlhd \pi_1(\Omega_{n_k}).$  This does not give the required $\pi_1(\Lambda_{k})\unlhd \pi_1(\Omega_{n_k}).$
An observation now is that instead of an arbitrary finite-sheeted regular covering of $\Gamma$ followed by the $\mathbb Z/2\mathbb Z$-homology covering
(that we alluded to above Theorem~\ref{thm:exist}),  we take finitely many times (same for each $k\in\mathbb N$) the $\mathbb Z/2\mathbb Z$-homology covering (first, to sufficiently enlarge the girth of the base graphs, then to produce walls). This amounts to take finitely many times the subgroup generated by all the squares of elements\footnote{For a group $K$, we denote by $K^{(2)}\unlhd K$ the subgroup generated by the squares of elements of $K$.} of the fundamental group of the base graph, cf.~\cite{AGS}. 
Such a subgroup is characteristic. Being characteristic is transitive. This yields  a characteristic subgroup $\pi_1(\Lambda_{k}):=((\pi_1(\Gamma_{n_k})^{(2)})^{(2)})^{\ldots (2)}\unlhd \pi_1(\Gamma_{n_k}),$ whence $\pi_1(\Lambda_{k})\unlhd \pi_1(\Omega_{n_k})$
as required. Thus,
$\Lambda_k$ is a Cayley graph for a group $H_k$ on $\vert S_{n_k}\vert$ generators. 

A surjective homomorphism $H_k\twoheadrightarrow G_{n_k}$ 
with the identity map on the generators $S_{n_k}$ does exist by construction. 
In detail, one can use Lemma~\ref{lem:cover} and an observation that a graphical presentation given by a Cayley graph of a group  defines this group itself,
that is,  $G(Cay(G_n, S_n))=G_n$ and $G(Cay(H_k, S_{n_k}))=H_k$~\cite[Example 1.2]{Dom}. 

Thus, each $\Gamma_{n_k}$ has two labelings, by $S$ and by $S_{n_k}$.
Since both labelings are reduced the corresponding metric spaces are both isometric to $\Gamma_{n_k}$ equipped with the edge-length metric.
The labeling of  $\Lambda_k$ is induced from the corresponding labeling of $\Gamma_{n_k}$ through the graph covering, whence
the same conclusion for $\Lambda_k.$ 
\end{proof}

\section{Wreath products}\label{sec:wp}

\begin{defn}[Restricted permutational wreath product]\label{defn:wp}
Let $A$ and $B$ be finitely generated groups and let $p\colon B\twoheadrightarrow Q$ be a surjective homomorphism. Then
the \emph{restricted permutational wreath product} of  $A$ and $B$ through $Q$ is the split extension
$$
 A\wr_{Q} B=\bigoplus_{Q}A\rtimes B,
$$
where 
$\bigoplus_{Q}A$ is the group of finitely supported functions $\phi\colon Q \to A$
with the pointwise multiplication, and $B$ acts on $\bigoplus_{Q}A$ by permuting the indices by multiplication on the left.

Let $U$ and $V$ be finite generating sets of $A$ and $B$, respectively. Then
$\{(\delta_u, 1_B )\colon u\in U\}\cup\{(\mathbf 1, v)\colon  v\in V\}$ generate $A\wr_{Q} B,$ where
$\delta_u(q)=u$ if $q=1_Q$ and
$\delta_u(q)=1_A$ otherwise, and $\mathbf 1(q)=1_A$ for all $q\in Q$.
We denote such a generating set briefly by $\{\delta_{u}\colon u\in U\}\cup V.$
 \end{defn}

We consider groups $G$ and $H$ given by Proposition~\ref{prop:surj}
applied to a dg-bounded expander which is a sequence of finite Cayley graphs; see, for instance, Example~\ref{ex:lps} for such an expander and Proposition~\ref{prop:surjC} for additional properties of such $G$ and $H$.
We denote by $S$ and $T$ their respective finite generating sets\footnote{We use a distinct letter $T$ for generators of $H$, for the generality of the argument.}.
Our group in Theorem~\ref{thm:main} is the restricted permutational wreath product of $\Z/2\Z$ and $H$ through its quotient $G$.
We denote this group by $$\G=\Z/2\Z\wr_G H.$$ 
We consider the following generating subset $\Sigma=\{\delta_{1_{G}}\}\cup T$ of $\G$. 

The Cayley graph $Cay(G,S)$ contains a sequence of subgraphs isometric to  Cayley graphs $Cay(G_n,S_n)_{n\in\mathbb{N}}$ forming a dg-bounded expander\footnote{e.g. with respect to $S_n$ with $|S_n|=p+1$ if we work with the expander from Example~\ref{ex:lps}.} and $Cay(H,T)$ contains a sequence of subgraphs isometric to Cayley graphs $Cay(H_n,T_n)_{n\in\mathbb{N}}$ such that the  restriction of the projection $H\twoheadrightarrow G$ to $H_n$ identifies with a surjective homomorphism $H_n\twoheadrightarrow G_n$ mapping $T_n\mapsto S_n$.  We shall denote $$\G_n=\Z/2\Z\wr_{G_n} H_n, \hbox{ and } \Sigma_n=\{\delta_{1_{G_n}}\}\cup T_n.$$

We have the following easy fact.
\begin{prop}\label{prop:subwreath}
Let $\pi\colon K\twoheadrightarrow K'$ be a surjective homomorphism of two finitely generated groups, let $V$ be a finite generating subset of $K$, and denote $V'=\pi(V)$. 
Assume that $Cay(K,V)$ contains a copy of some Cayley graph $Cay(L,U)$ as a subgraph, that there exists a group $L'$ and a bijection $f\colon\pi(L)\to L'$, such that $f\circ \pi$ defines a surjective homomorphism $L\twoheadrightarrow L'$. 
Then $Cay(\Z/2\Z\wr_{L'} L,\{\delta_{1_{L'}}\}\cup U)$ embeds as a subgraph of $Cay(\Z/2\Z\wr_{K'} K,\{\delta_{1_{K'}}\}\cup V)$.
 \end{prop}
 \begin{proof}
To simplify notation, we identify $Cay(L,U)$ (resp.\ $Cay(L',U')$)  as a subgraph of $Cay(K,V)$  (resp.\ of $Cay(K',V')$).
Let us first describe the embedding at the level of vertices.  To every pair $(\phi,l)\in \bigoplus_{L'}\mathbb Z/2\mathbb Z\rtimes L$, we associate the element 
$(\Phi,l)\in \bigoplus_{K'}\mathbb Z/2\mathbb Z\rtimes K$, where $\Phi=\phi$ in restriction to $L'$, and $\Phi=0$ elsewhere. 
Two vertices $(\phi_1,l_1)$ and $(\phi_2,l_2)$ are linked by an edge in $Cay(\Z/2\Z\wr_{L'} L,\{\delta_{1_{L'}}\}\cup U)$ if and only if either $l_1=l_2$ and $\phi_1$ and $\phi_2$ differ exactly at $\pi(l_1)$, or if $\phi_1=\phi_2$, and $l_1$ and $l_2$ are neighbors in $(L,U)$. Since an analogous statement holds for two vertices of $Cay(\Z/2\Z\wr_{K'} K,\{\delta_{1_{K'}}\}\cup V)$, we deduce that $(\Phi_1,l_1)$ and $(\Phi_2,l_2)$ are also neighbors in this Cayley graph. 
 \end{proof}

 \begin{cor}\label{cor:subwreath}\label{cor:subgraph}
 The graphs $Cay(\G_n,\Sigma_n)_{n\in\mathbb{N}}$ embed as subgraphs of $Cay(\G,\Sigma)$.
 \end{cor}
 
\section{No coarse embedding into a Hilbert space}
This section is dedicated to the proof that $Cay(\G,\Sigma)$ does not coarsely embed into a Hilbert space. This relies on the following {\it relative} Poincar\'e inequality satisfied by the sequence of graphs $Cay(\G_n,\Sigma_n)_{n\in\mathbb{N}}$. 
In what follows, $\HH$ denotes a Hilbert space, $\|\cdot \|_2$ its norm, and $B(v, r)$ the ball of radius $r>0$ centered at $v\in\HH.$

\begin{thm}\label{thm:Relpoincare}
There exists a constant $C>0$ such that every function $f\colon\G_n\to \HH$ satisfies
$$\frac{1}{|X_n|}\sum_{(x,y)\in \G_n\times X_n}\|f(x)-f(xy)\|^2_2\leqslant C\sum_{(x,\sigma)\in \G_n\times \Sigma_n}\|f(x)-f(x\sigma)\|^2_2,$$
where $X_n=\{\delta_g, \; g\in G_n\}\hookrightarrow W_n$.
\end{thm}

\begin{cor}\label{cor:noCE}
There exists a constant $D>0$ such that every $1$-Lipschitz map $f\colon\G\to \HH$ satisfies
$$\sup_{v\in \HH}|f^{-1}(B(v,D))|=\infty.$$ 
In particular, $\G$ does not coarsely embed into a Hilbert space.
\end{cor}
\begin{proof}[Proof of Corollary~\ref{cor:noCE}]
By Corollary \ref{cor:subgraph}, we view $Cay(\G_n,\Sigma_n)$ as a subgraph of $Cay(\G,\Sigma)$. Therefore, $f$ induces a $1$-Lipschitz map from $\G_n$ to $\HH$.
We deduce from Theorem \ref{thm:Relpoincare} that there exists a constant $C>0$ such that
$$\frac{1}{|X_n||W_n|}\sum_{(x,y)\in \G_n\times X_n}\|f(x)-f(xy)\|^2_2\leqslant C|\Sigma_n|,$$
from which we deduce that there exists $x\in W_n$ such that 
$$\frac{1}{|X_n|}\sum_{y\in X_n}\|f(x)-f(xy)\|^2_2\leqslant C|\Sigma_n|.$$
The sizes $|\Sigma_n|$ are bounded above, uniformly over $n\in\mathbb N$, e.g. by $|\Sigma|.$ Take a constant $D>0$, independent of $n$, such that $D\geqslant \sqrt{2C|\Sigma_n|}$. 
The above inequality implies that there exists a subset $Y_n\subseteq W_n$ with $|Y_n|\geqslant |X_n|/2$ and such that for all $y\in Y_n$, $f(y)$ lies in $B(f(x),D)\subseteq \HH$. Since the cardinality of $Y_n$ tends to infinity, this proves the corollary.
 \end{proof}
 
 \
 
\begin{proof}[Proof of Theorem~\ref{thm:Relpoincare}]
Consider the right regular representation $\rho_n$ of the group $\G_n$ on $\ell^2(\G_n,\HH)$, i.e.\ $\rho_n(w)f=f(\cdot w)$ for $w\in \G_n$. Then the relative Poincar\'e inequality can be rewritten as 
$$\frac{1}{|X_n|}\sum_{y\in X_n}\|f-\rho_n(y)f\|_2^2\leqslant C\sum_{\sigma\in \Sigma_n}\|f-\rho_n(\sigma)f\|_2^2.$$ 
Note that $w\mapsto \psi(w)=\|f-\rho_n(w)f\|_2^2$ defines a conditionally negative definite function\footnote{See e.g.~\cite[Section 1.1 and Section 2.1]{CCJ} for the definition and the characterization of the Haagerup property using such functions.} on $\G_n$. Hence the previous inequality, and therefore Theorem \ref{thm:Relpoincare}, is a consequence of the following inequality, valid for all conditionally negative definite functions $\psi$ on $\G_n$
\begin{equation}\label{eq:CNineq}
\sup_{y\in X_n}\psi(y)\leqslant C\sup_{\sigma\in \Sigma_n}\psi(\sigma).
\end{equation}
On the other hand, let us interpret the fact that  $Cay(G_n,S_n)_{n\in\mathbb{N}}$ is an expander in terms of positive definite functions on $G_n$. Set $r=|S_n|$ and  consider the projection from the free group  
$F_r\twoheadrightarrow G_n$, whose kernel is denoted by $N_n$.  Recall that  $Cay(G_n,S_n)_{n\in\mathbb{N}}$ being an expander is equivalent to the fact that $F_r$  has Property $(\tau)$ with respect to the sequence $(N_n)_{n\in\mathbb{N}}$~\cite[Theorem 4.3.2]{Lub}. This condition says that the representations factoring through some $G_n$ admit a uniform, over $n\in\mathbb N,$ Kazhdan constant.  
 In terms of positive definite functions, this implies the existence of  $\eps>0$ and $\delta>0$ such that for every $n$, and every positive definite function $\phi$ on $G_n$ such that 
 $\inf_{s\in S_n} |\phi(s)|\geqslant 1-\delta$, we have $\inf_{g\in G_n}|\phi(g)|\geqslant \eps$.
Therefore, (\ref{eq:CNineq}) follows from Proposition \ref{prop:semi direct} below.

 \end{proof}

\begin{prop}\label{prop:semi direct}
Given $\eps>0$ and $\delta>0$, there exists $C>0$ such that the following holds.
Let $A$ be an abelian group, $K_0$ a group, and consider a homomorphism $\alpha\colon K_0\to \Aut(A)$. Consider another group $K$ with a surjective homomorphism $\pi\colon K\twoheadrightarrow K_0$,  and denote $\bar{\alpha}=\alpha \circ \pi$. Let $U$ be a generating subset of $K$, and let $U_0=\pi(U).$
 Assume that for every positive definite function $\phi$ on $K_0$ satisfying $\inf_{u\in U_0} |\phi(u)|\geqslant 1-\delta$, we have $\inf_{k\in K_0}|\phi(k)|\geqslant \eps$.
Then, for every $a\in A$, every generating subset $V$ of $L=A\rtimes_{\bar{\alpha}}K$ containing $a$ and $U$, and every conditionally negative definite function $\psi$ on $L$, we have
$$\sup_{y\in \bar{\alpha}(K)(a)}\psi(y)\leqslant C\sup_{v\in V}\psi(v).$$ 
\end{prop}
\begin{proof}

Let us start with a restatement of \cite[Theorem 3.1]{CI} in terms of positive definite functions.
\begin{lem}
Given $\eps>0$ and $\delta>0$, there exist $\eps'$ and $\delta'>0$ such that the following holds. Let $A$ be an abelian group, $K_0$ a group, and consider a homomorphism $\alpha\colon K_0\to \Aut(A)$. Consider another group $K$ with a surjective homomorphism $\pi\colon K\twoheadrightarrow K_0$,  and denote $\bar{\alpha}= \alpha\circ \pi$. Let $U$ be a generating subset of $K$, and let $U_0=\pi(U).$
 Assume that for every positive definite function $\phi$ on $K_0$ satisfying $\inf_{u\in  U_0} |\phi(u)|\geqslant 1-\delta$, we have $\inf_{k\in K_0}|\phi(k)|\geqslant \eps$.
Then, for every $a\in A$, every generating subset $V$ of $L=A\rtimes_{\bar{\alpha}}K$ containing $a$ and $U$, and every positive definite function $\phi'$ on $L$ satisfying $\inf_{v\in V} |\phi'(v)|\geqslant 1-\delta'$, we have $\inf_{y\in \bar{\alpha}(K)(a)}|\phi'(y)|\geqslant \eps'$.
\end{lem}
The fact that $\eps'$ and $\delta'$ only depend on $\eps$ and $\delta$, and not on the particular group $L$ is important for our purposes: this follows from  the proof of \cite[Theorem 3.1]{CI}.  In order to conclude, we use the following classical lemma.
\begin{lem}\label{lem:Schoenberg}
Let $\eps>0$ and $\delta>0$. Let $F$ be a group and $V$ be a finite generating subset of $F$, and $X$ be a subset of $F$. Assume that for every positive definite function $\phi$ on $F$ 
satisfying $\inf_{v\in V} |\phi(v)|\geqslant 1-\delta$, we have $\inf_{x\in X}|\phi(x)|\geqslant \eps$. Then every conditionally negative definite function $\psi$ on $F$ satisfies 
$$\sup_{x\in X}\psi(x)\leqslant \frac{-\log \eps}{\delta}\sup_{v\in V}\psi(v).$$
\end{lem}
\begin {proof}
Clearly we can assume that $\psi$ is non-zero. In particular, it is non-zero on the generating set $V$. 
We normalize $\psi$ so that $\sup_{v\in V}\psi(v)= 1$, and
consider the positive definite function $\phi=\exp(-\delta\psi)$. 
We have that $\inf_{v\in V} |\phi(v)|\geqslant e^{-\delta}\geqslant 1-\delta$. Therefore, 
$\inf_{x\in X}|\phi(x)|\geqslant \eps$,  from which we deduce the lemma.
\end{proof}

Applying Lemma \ref{lem:Schoenberg} to $F=L$ now yields the conclusion of Proposition \ref{prop:semi direct} with $C=\frac{-\log \eps'}{\delta'}.$
\end{proof}

\section{Further questions and conjectures}

Our main result, Theorem~\ref{thm:main}, provides a finitely generated example. The next question is natural 
and has also an interest in the setting of the Novikov conjecture.  Observe that being an extension of two groups with the Haagerup property, 
our group $\Z/2\Z\wr_G H$ satisfies
the strong Baum-Connes conjecture (which is strictly stronger than the Baum-Connes conjecture with coefficients~\cite{MaNe}).

\begin{que}
Does there exist a finitely \emph{presented} group that does not coarsely embed into a Hilbert space and yet does not contain a weakly embedded expander?
\end{que}

In positive direction, 
\begin{que}[\cite{DG}]
Does a central extension of $\Z$ by a coarsely embeddable group admit a coarse embedding into a Hilbert space?
\end{que}

From the quantitative side, given a dg-bounded expander $\Gamma=(\Gamma_n)_{n\in\mathbb{N}}$, 
Hume proves the existence of $2^{\aleph_0}$ expanders $\Gamma^r=(\Gamma_n^r)_{n\in\mathbb{N}},\ r\in\mathbb R,$ such that
there is no \emph{regular}\footnote{A map between graphs is \emph{regular} if it is Lipschitz and pre-images of vertices have uniformly bounded cardinality. It is a far generalization of coarse, and so, of quasi-isometric, embeddings.} map $\Gamma^r\to\Gamma^{r'}$ for $r\not=r'$~\cite{Hume}.  Namely, given $M,N\subseteq \mathbb N$ with $N\setminus M$ infinite, he shows 
(up to passing to a subsequence so that $\vert \Gamma_{n+1}\vert / \vert \Gamma_n\vert\to\infty$  as $n\to\infty$) that
there is no regular map from $\Gamma^N=(\Gamma_n)_{n\in N}$ to $\Gamma^M=(\Gamma_n)_{n\in M}.$ 
We observe that taking finite-sheeted graph coverings as in Proposition~\ref{prop:surjC} provides $2^{\aleph_0}$ graphs with walls such
that there is no regular map from $\Lambda^N=(\Lambda_n)_{n\in N}$ to $\Lambda^M=(\Lambda_n)_{n\in M}.$ 

Taking $G(\Gamma_N)$ defined by $\Gamma_N$
 (up to passing to a subsequence such that  $\gi\Gamma_{n+1}>2\vert \Gamma_n\vert$ as $n\to\infty$), 
gives $2^{\aleph_0}$ special Gromov monsters $G_r,\ r\in\mathbb R,$ such that there is no
regular (hence, coarse or quasi-isometric embedding) map $G_r\to G_{r'}$ whenever $r\not =r'$~\cite[Theorem 2.9]{Hume}.
Inspired by this result and our observation above, we have the following

\begin{conj}
There exists $2^{\aleph_0}$ regular
equivalence classes of Haagerup monsters. 
\end{conj}

\begin{conj}
There exists $2^{\aleph_0}$ regular equivalence classes of counterexamples as constructed in our Theorem~\ref{thm:main} and Theorem~\ref{thm:mainbis}. 
\end{conj}


\bigskip
\footnotesize


\begin{bibdiv}
\begin{biblist}

\bib{AR}{book}{
    author= {Anantharaman-Delaroche, G.},
     author= {Renault, J.},
   title = {Amenable groupoids},
    series={Monographies de L'Enseignement Mathematique},
     volume={36},
    year = {2000},
    publisher = {L'Enseignement Mathematique},
    place= {Geneva},
    pages= {x+196},
}


\bib{AD}{article}{
   author={{Arzhantseva}, G.},
   author={{Delzant}, T.},
    title={Examples of random groups},
 year={2008},
 journal={preprint, available on authors' websites},
}




\bib{AGS}{article}{
   author={Arzhantseva, G.},
   author={Guentner, E.},
   author={{\v{S}}pakula, J.},
   title={Coarse non-amenability and coarse embeddings},
   journal={Geom. Funct. Anal.},
   volume={22},
   date={2012},
   number={1},
   pages={22--36},
}

\bib{AO}{article}{
   author={{Arzhantseva}, G.},
   author={{Osajda}, D.},
    title={Graphical small cancellation groups with the Haagerup property},
  journal = {arXiv:1404.6807},
 year={2014},
}

\bib{AT}{article}{
   author={{Arzhantseva}, G.},
   author={{Tessera}, R.},
    title={Relative expanders},
  journal={Geom. Funct. Anal.},
   volume={25},
   date={2015},
   number={2},
   pages={317--341},
   }


\bib{CD}{article}{
   author={Cave, Ch.},
   author={Dreesen, D.},
   title={Embeddability of generalised wreath products},
   journal={Bull. Aust. Math. Soc.},
   volume={91},
   date={2015},
   number={2},
   pages={250--263},
}


\bib{CCJ}{book}{
    author= {Cherix, P.-A.},
     author= {Cowling, M.},
     author= {Jolissaint, P.},
     author= {Julg, P.},
     author= {Valette, A.},
    title = {Groups with the Haagerup property (Gromov's a-T-menability)},
    series={Progress in Mathematics},
     volume={197},
    year = {2001},
    publisher = {Birkha\"user Verlag},
    place= {Basel},
    pages= {x+207},
}

  



\bib{CI}{article}{
   author={Chifan, I.},
   author={Ioana, A.},
   title={On Relative property (T) and Haagerup's property},
   journal={Transactions of the AMS},
   volume={363},
   date={2011},
   pages={6407--6420},



}





\bib{CSV}{article}{
   author={de Cornulier, Y.},
   author={Stalder, Y.},
   author={Valette, A.},
   title={Proper actions of wreath products and generalizations},
   journal={Trans. Amer. Math. Soc},
   volume={364},
   date={2012},
   pages={3159--3184},
}


\bib{DG}{article}{
   author={Dadarlat, M.},
   author={Guentner, E.},
   title={Constructions preserving Hilbert space uniform embeddability of
   discrete groups},
   journal={Trans. Amer. Math. Soc.},
   volume={355},
   date={2003},
   number={8},
   pages={3253--3275},
}






\bib{Gr_ai}{article}{
   author={Gromov, M.},
   title={Asymptotic invariants of infinite groups},
   conference={
      title={Geometric group theory, Vol.\ 2},
      address={Sussex},
      date={1991},
   },
   book={
      series={London Math. Soc. Lecture Note Ser.},
      volume={182},
      publisher={Cambridge Univ. Press, Cambridge},
   },
   date={1993},
   pages={1--295},
}


\bib{Gr_sp}{article}{
   author={Gromov, M.},
   title={Spaces and questions},
   note={GAFA 2000 (Tel Aviv, 1999)},
   journal={Geom. Funct. Anal.},
   date={2000},
   pages={118--161},
}

\bib{Gr_rw}{article}{
   author={Gromov, M.},
   title={Random walk in random groups},
   journal={Geom. Funct. Anal.},
   volume={13},
   date={2003},
   number={1},
   pages={73--146},
}

\bib{Dom}{article}{
   author={Gruber, D.},
   title={Infinitely presented graphical small cancellation groups: Coarse embeddings, acylindrical hyperbolicity, and subgroup constructions},
   journal={PhD Thesis, University of Vienna},
   date={2015},
}




\bib{GK}{article}{
   author={Guentner, E.},
   author={Kaminker, J.}, 
   title={Geometric and analytic properties of
groups},
   journal={in: Noncommutative geometry, Ed. by S. Doplicher and R. Longo, Lecture Notes Math.},
   volume={1831},
   date={2004},
   pages={253--262},
}

\bib{HLS}{article}{
   author={Higson, N.},
   author={Lafforgue, V.},
   author={Skandalis, G.},
   title={Counterexamples to the Baum-Connes conjecture},
   journal={Geom. Funct. Anal.},
   volume={12},
   date={2002},
   number={2},
   pages={330--354},
}



\bib{Hume}{article}{
   author = {{Hume}, D.},
    title={A continuum of expanders},
  journal= {arXiv:1410.0246},
     year={2014},
}


\bib{KYu}{article}{
   author={Kasparov, G.},
   author={Yu, G.},
   title={The Novikov conjecture and geometry of Banach spaces},
   journal={Geom. Topol.},
   volume={16},
   date={2012},
   number={3},
   pages={1859--1880},
}



\bib{KW}{article}{
   author={Kirchberg, E.},
   author={Wassermann, S.},
   title={Permanence properties of $C^*$-exact groups},
   journal={Doc. Math.},
   volume={4},
   date={1999},
}




\bib{Li}{article}{
   author={Li, S.},
   title={Compression bounds for wreath products},
   journal={Proc. Amer. Math. Soc.},
   volume={138},
   date={2010},
   number={8},
   pages={2701--2714},
}




\bib{Lub}{book}{
   author={Lubotzky, A.},
   title={Discrete groups, expanding graphs and invariant measures},
   series={Progress in Mathematics},
   volume={125},
   note={With an appendix by Jonathan D. Rogawski},
   publisher={Birkh\"auser Verlag},
   place={Basel},
   date={1994},
   pages={xii+195},
}

\bib{LPS}{article}{
   author={Lubotzky, A.},
   author={Phillips, R.},
   author={Sarnak, P.},
   title={Ramanujan graphs},
   journal={Combinatorica},
   volume={8},
   date={1988},
   number={3},
   pages={261--277},
}

\bib{Ma}{article}{
   author={Margulis, G. A.},
   title={Explicit group-theoretic constructions of combinatorial schemes
   and their applications in the construction of expanders and
   concentrators},
   language={Russian},
   journal={Problemy Peredachi Informatsii},
   volume={24},
   date={1988},
   number={1},
   pages={51--60},
   translation={
      journal={Problems Inform. Transmission},
      volume={24},
      date={1988},
      number={1},
      pages={39--46}},
}

\bib{M}{article}{
   author={Matous\v{e}k, J.}, 
   title={On embedding expanders
into $\ell\sb p$ spaces},
   journal={Israel J. Math.},
   volume={102},
   date={1997},
   pages={189--197},

}
%
%
\bib{MaNe}{article}{
   author={Mayer, R.},
   author={Nest, R.}, 
   title={The Baum-Connes conjecture via localisation of categories},
   journal={Topology},
   volume={45},
   number={2},
   date={2006},
   pages={209--259},

}

\bib{NYu}{book}{
    author= {Nowak, P.},
     author= {Yu, G.},
    title = {Large scale geometry},
    pages = {xiv + 189},
    year = {2012},
    publisher = {Z\"urich: European Mathematical Society (EMS)},
    }
 
 \bib{Oll}{article}{
   author={Ollivier, Y.},
   title={On a small cancellation theorem of Gromov},
   journal={Bull. Belg. Math. Soc. Simon Stevin},
   volume={13},
   date={2006},
   number={1},
   pages={75--89},
}



\bib{Os}{article}{
   author={{Osajda}, D.},
    title={Small cancellation labellings of some infinite graphs and applications},
 journal={arXiv:1406.5015},
     year= {2014},
}
 
 
    
\bib{R}{book}{
   author={Roe, J.},
   title={Lectures on coarse geometry},
   series={AMS University Lecture Series},
   volume={31},
   date={2003},
}



\bib{Sela}{article}{
   author={Sela, Z.},
   title={Uniform embeddings of hyperbolic groups in Hilbert spaces},
   journal={Israel J. Math.},
   volume={80},
   date={1992},
   number={1-2},
   pages={171--181},
}


\bib{Yu}{article}{
   author={Yu, G.},
   title={The coarse Baum-Connes conjecture for spaces which admit a uniform
   embedding into Hilbert space},
   journal={Invent. Math.},
   volume={139},
   date={2000},
   number={1},
   pages={201--240},
}

\end{biblist}
\end{bibdiv}
\end{document}